\documentclass[a4paper,12pt]{article}
\usepackage{On_the_Sn_problem}
\usepackage[onehalfspacing]{setspace}

\begin{document}

\title{On the $\sn$-Problem}
\date{\today}
\author{S\"oren Christensen\thanks{Mathematisches Seminar, Christian-Albrechts-Universit\"at zu Kiel, Ludewig-Meyn-Str.  4, D-24098 Kiel, Germany, E-mail:  christensen@math.uni-kiel.de}, Simon Fischer\thanks{Mathematisches Seminar, Christian-Albrechts-Universit\"at zu Kiel, Ludewig-Meyn-Str.  4, D-24098 Kiel, Germany, E-mail:  fischer@math.uni-kiel.de}}

\maketitle
\begin{abstract}
    The Chow-Robbins game is a classical still partly unsolved stopping problem introduced by Chow and Robbins in 1965.
    You repeatedly toss a fair coin. After each toss, you decide if you take the fraction of heads up to now as a payoff, otherwise you continue.
    As a more general stopping problem this reads  
        \[V(n,x) = \sup_{\tau  }\e \left [ \frac{x + S_\tau}{n+\tau}\right]\]
    where $S$ is a random walk.
    We give a tight upper bound for $V$ when $S$ has subgaussian increments by using the analogous time continuous problem with a standard Brownian motion as the driving process. 
    For the Chow-Robbins game we as well give a tight lower bound and use these to calculate, on the integers, the complete continuation and the stopping set of the problem for $n\leq 489.241$. 
\end{abstract}

\clearpage

\section{Introduction}
We repeatedly toss a fair coin. After each toss we can take the proportion of \textit{heads} up to now as our reward, or continue. This is known as the Chow-Robbins game.
It was first presented by Yuan-Shih Chow and Herbert Robbins in 1965 \cite{chow1965}. As a stopping problem this formulates as 
\begin{equation}\label{eq:vsdef}
    V^{S}(t,x) = \sup_{\tau  }\e \left [ \frac{x + S_\tau}{t+\tau}\right],
\end{equation}
where $S$ is a random walk. In the classical version $S$ has symmetric Bernoulli increments but it is possible to take different random walks as well. Chow and Robbins showed that an optimal stopping time exists in the Bernoulli case, later Dvoretzky \cite{dvoretzky1967} proved this for general centered iid.\ increments with finite variance. But it was (and to some extent still is) difficult to see how that solution looks like.
Asymptotic results were given by Shepp in 1969 \cite{shepp1969}, who showed that the boundary of the continuation set $\partial C$  can be written as a function $b:\RR^{+}\to \RR$ with
    \[\lim_{t\to\infty}\frac{b(t)}{\alpha \sqrt{t}} = 1.\]
Here $\alpha\sqrt{t}$ is the boundary of the analogous stopping problem for a standard Brownian motion $W$ (see Lemma \ref{lem:wt})
\begin{equation}\label{eq:vwdef}
    V^W(t,x)=\sup_{\tau  }\e \left [ \frac{x+W_\tau}{t+\tau}\right ].
\end{equation}

In 2007 Lai, Yao and AitSahlia \cite{lai2007} gave a second order approximation for the limit of $b$, that is
    \[\lim_{t\to\infty}\big (\alpha \sqrt{t}-b(t) \big)= \1.\]
Lai and Yao \cite{lai2006} also calculated some approximation for values of $b$ by using the value function \eqref{eq:vwdef}, without constructing it as an upper bound. They as well gave some calculations for a random walk with standard normal increments.\\
A more rigorous computer analysis was given by H\"aggstr\"om and W\"astlund in 2013 \cite{haeggstoem2013}.
Using backward induction from a time horizon $T=10^{7}$, they calculated lower and upper bounds for $V^{S}$. For points, reachable from $(0,0)$, they calculated if they belong to the stopping or to the continuation set and were able to do so for all but 7 points $(n,x)$ with $n\leq 1000$.\\

In this paper we will give much sharper upper and lower bounds for $V^{S}$. Using backward induction with these bounds, we are able to calculate all stopping and continuation points $(n,x)\in \NN\times \ZZ$ with $n\leq 10^{5}$. We show that all 7 points $(n,x)$ with $n\leq 1000$ that were left open in \cite{haeggstoem2013}, belong to the stopping set.\\
In Section \ref{sec:up} we construct an upper bound for the value function \eqref{eq:vsdef}
for random walks with subgaussian increments. The main observation is that the value function 
\eqref{eq:vwdef}
is an upper bound for $V^{S}$. 
This is carried out in Subsection \ref{ssec:sn} on the Chow-Robbins game.
In Subsection \ref{ssec:gen} we discuss how this kind of upper bound can be constructed for more general gain functions $g$ and 
   \[V^{S}(t,x) = \sup_{\tau  }\e \left [ g(t+\tau,x + S_\tau)\right].\]
In Section \ref{sec:low} we construct a lower bound for $V^{S}(T,x)$ in the Bernoulli case for a given time horizon $T$. We show that there exists $0<c<1$ and $K>0$ such that 
    \[K \int_{0}^{\infty}e^{ax-\frac{c}{2}a^2T} \dif a\leq V^{S}(T,x)\]
for all $x\leq b(T)$. We then show that the relative error of the bounds is of order $\mathcal{O}\left(\frac{1}{T}\right)$ for $x\geq0$.
In Section \ref{sec:comp} we give computational results for the Chow-Robbins game in the Bernoulli case and give a detailed description of our methods. We calculate all integer valued points in the stopping and in the continuation set for $n\leq 10^{5}$ and give some examples how $V^{S}(t,x)$ and $b(t)$ look for continuous $t$ close to zero.

\paragraph{Notation}
We want to introduce some notation, that we are going to use. $V$ denotes value functions, $D = \{V=g\} = \{(t,x)\mid V(t,x) = g(t,x)\}$ the corresponding stopping set and $C = \{V>g\}$ the continuation set. With $(t,x)$ we denote real variables, with $n$ positive integers.
With a superscript we denote which driving process is used, e.g.\
$C^{S} = \{V^{S} = g\}$, $C^{W} = \{V^{W} = g\}$, etc.
$V_{u}$ and $V_{l}$ denote upper and lower bounds resp.

\section{An upper bound for the value function $V^{S}$}\label{sec:up}
We construct an upper bound for the value function $V^{S}$ of the $\sn$-problem, where $S$ can be any random walk with subgaussian increments. The classical Chow-Robbins game is a special case of these stopping problems.

\begin{definition}[subgaussian random variable]
    Let $\sigma^2>0$. A real, centered random variable $\xi$
    is called 
    \textbf{$\sigma^2$-subgaussian} (or subgaussian with parameter $\sigma^2$), if
    \[ \e[e^{a\xi}] \leq e^{\frac{\sigma^2 a^2}{2}} ~~\text{for all } a\in \RR.\]
\end{definition}
\noindent Some examples of subgaussian random variables are:
\begin{itemize}
    \item  $X$ with $P(X_i =-1) =P(X_i =1) = \frac{1}{2} $ is 1-subgaussian,
    \item The normal distribution $\mathcal{N}(0,\sigma^2)$ is $\sigma^2$-subgaussian,
    \item The uniform distribution on $[-a,a]$ is $a^2$-subgaussian,
    \item Any random variable $Y$ with values in a compact interval $[a,b]$, is $\frac{(b-a)^2}{4}$-subgaussian.
\end{itemize}

 In the following we show that the value function $V^{W}$ of the continuous time problem \eqref{eq:vwdef} is an upper bound for the value function $V^{S}$, whenever $S$ is a random walk with 1-subgaussian increments. We first state the solution of \eqref{eq:vwdef}.
 
\paragraph{The solution of the continuous time problem}
Let
\begin{equation}\label{eq:vw3}
    h(t,x) := (1-\alpha^2)\int_{0}^{\infty}e^{ax-\frac{a^2}{2}t} \dif a 
    =(1-\alpha^2)\frac{1}{t}\frac{\Phi_{t}(x)}{\varphi_t(x)},
\end{equation}
where $\Phi_t(x) = \Phi(x/\sqrt{t})$ denotes the cummulative distribution function of a centered normal distribution with variance $\sigma^2 = t$, $\varphi_t$ the corresponding density function and $\alpha \approx 0.839923675692373$ is the unique solution to
\begin{equation*}
    \alpha \varphi(\alpha) = (1-\alpha^2)\Phi(\alpha).
\end{equation*}
  
\begin{lemma}\label{lem:wt}
    The stopping problem \eqref{eq:vwdef} is solved by 
        \[\tau_{\ast} = \inf\{s\mid x+W_s\geq \alpha\sqrt{s+t}\}\]
    with value function
        \[V^W(t,x) = \begin{cases}
                    h(t,x) & \text{if } x \leq \alpha               \sqrt{t},\\
                    \frac{x}{t} & \text{else}.
                    \end{cases}\]
    $V^{W}(t,x)$ is differentiable (smooth fit), and $h(t,x)\geq g(t,x) = \frac{x}{t}$ for all $t>0$ and $x\in \RR$.             
\end{lemma}
This result has first been proven independently by Shepp \cite{shepp1969} and Walker \cite{walker1969}.\\

\paragraph{Proof that $V^{W}$ is an upper bound for $V^{S}$} 
We know from general theory that $V^{S}$ is the smallest superharmonic function dominating the gain function $g$, see e.g.\ \cite{peskir06}.
If we find a superharmonic function dominating $g$ we have an upper bound for $V^{S}$.

\begin{lemma}\label{lem:mart}
Let $X_i$ be iid.\ 1-subgaussian random variables, $S_n = \sum_{i=1}^{n}X_i$.
The function
    \[h:\RR^{+}\times \RR \to \RR^{+}, (t,x)\mapsto (1-\alpha^{2})\int_{0}^{\infty} e^{ax - \frac{1}{2}a^2t}\dif a\]
is $S$-superharmonic.

\end{lemma}

\begin{proof}
    We first show the claim for fixed $a\in \RR$ and
        \[f(t,x)=e^{ax - \frac{1}{2}a^2t}.\]
    We need to show that $\e [f(t+1,x+X_1)]\leq f(t,x)$, for all $t>0$ and $x\in \RR$, and calculate
    \begin{align}
        &\e[e^{a(x+X_{1}) - \frac{a^2}{2}(t+1)}] \leq e^{ax - \frac{a^2}{2}t}\nonumber \\ 
        \Leftrightarrow &  e^{ax - \frac{a^2}{2}(t+1)} \e[e^{aX_{1}}]\leq e^{ax - \frac{a^2}{2}t}\nonumber \\
        \Leftrightarrow & e^{-\frac{a^2}{2}} \e[e^{aX_{1}}] \leq 1\nonumber \\
        \Leftrightarrow & \e[e^{aX_{1}}] \leq e^{\frac{a^2}{2}}.\label{sg2}
    \end{align}
    The last inequality \eqref{sg2} is just the defining property of a 1-subgaussian random variable. \\
    By integration over $a$ and multiplication with $(1-\alpha)$ the result follows.
\end{proof}

\begin{theorem}[An upper bound for $V^{S}$]\label{th:up}
    Let $W$ be a standard Brownian motion, $S$ a random walk with 1-subgaussian increments and
    \begin{align}
        &V^{S}(t,x) = \sup_{\tau  }\e \left [ \frac{x + S_\tau}{t+\tau}\right],\label{eq:sn5} \\
        & V^W(t,x)=\sup_{\tau  }\e \left [ \frac{x+W_\tau}{t+\tau}\right ].\nonumber
    \end{align}
    Then
        \[V^{W}(t,x)\geq V^{S}(t,x), \text{ for all $t>0$, $x\in \RR$.}\]
\end{theorem}

\begin{proof}
    Let $h(t,x)=(1-\alpha^{2})\int_0^{\infty}e^{ax - \frac{a^{2}}{2}t}\dif a$ as in Lemma \ref{lem:wt}. We know that
    \begin{itemize}
        \item $h\geq g$ (Lemma \ref{lem:wt}),
        \item $h$ is $S$-superharmonic (Lemma \ref{lem:mart}),
        \item $V^{W} = h \mathbb{I}_{C^{W}} + g \mathbb{I}_{D^{W}}$.
    \end{itemize}
    $V^{S}$ is the smallest superharmonic function dominating $g$, therefore $V^{S}\leq h$.
    We know from Lemma \ref{lem:wt} that 
        \[V^{W}(t,\alpha\sqrt{t}) = h(t,\alpha\sqrt{t}) = g(t,\alpha\sqrt{t}),\]
    and therefore
        \[g(t,\alpha\sqrt{t})\geq V^{S}(t,\alpha\sqrt{t}) \leq h(t,\alpha\sqrt{t}) = g(t,\alpha\sqrt{t}).\]
    Hence $V^{S}(t,\alpha\sqrt{t}) = g(t,\alpha\sqrt{t})$ and $(t,\alpha\sqrt{t})\in D^{S}$.
    The boundary $\partial C^{S}$ is the graph of a function, therefore  $(t,x)\in D^{S}$ for all $x\geq \alpha\sqrt{t}$. It follows that $C^{S}\subset C^{W}$ and that
    $V^{S}(t,x)\leq V^{W}(t,x)$, for all $t>0$, $x\in \RR$.
 \end{proof}

\begin{corollary}
    From the proof we see that 
        \[C^{S}\subset C^{W},\]
    and 
        \[b(t)\leq \alpha\sqrt{t}, \text{ for all } t>0.\]
\end{corollary}

\subsection{The Chow-Robbins game}\label{ssec:sn}
Let $X_1, X_2 \dots$ be iid.\ random variables with 
$P(X_i =-1) =P(X_i =1) = \frac{1}{2} $ and
$S_n = \sum_{i=1}^{n}X_i$.
The classical Chow-Robbins problem is given by
\begin{equation}\label{eq:cr}
    V^S(t,x)=\sup_{\tau  }\e \left [ \frac{x + S_\tau}{t+\tau}\right ].
\end{equation}
The $X_i$ are 1-subgaussian with variance 1. An a.s.\ finite stopping time $\tau_{\ast}$ exists that solves \eqref{eq:cr}, see \cite{dvoretzky1967}. By Theorem \ref{th:up} we get that
    \[V^S(t,x)\leq V^{W}(t,x) =  h(x,t)\mathbb{I}_{\{x \leq \alpha \sqrt{t}\}} + \frac{x}{t}\mathbb{I}_{\{x > \alpha\sqrt{t}\}}. \]
We will see later on that this upper bound is very tight.
We will construct a lower bound for $V^S$ in the next section and
give rigorous computer analysis of the problem in Section \ref{sec:comp}.

\begin{example}
    For some time it was unclear whether it is optimal to stop in $(8,2)$ or not. It was first shown in \cite{medina2009} and later confirmed in \cite{haeggstoem2013} that $(8,2)\in D^{S}$.\footnote{In \cite{haeggstoem2013} this is written as $5-3$, 5 \textit{heads}$-$3 \textit{tails}.} We show how to immediately prove this with our upper bound.\\
    We choose the time horizon $T=9$, set $V_{u}^{S}(T,x) = V^{W}(T,x)$ and calculate with one-step backward induction $V_{u}^{S}(8,2)$ as an upper bound for $V^{S}(8,2)$:
    \begin{align*}
        &  V^{S}_u(9,3) = \frac{3}{9}=\frac{1}{3} & \text{since }3> \alpha\sqrt{9}\\
        &  V^{S}_u(9,1) = h(9,1) \approx 0.1642
    \end{align*}
    and we get
    \begin{align*}
        & V^{S}(8,2)\leq V_{u}^{S}(8,2) =\max\left\{\frac{2}{8}, \frac{V^{S}_u(9,3) +  V^{S}_u(9,1)}{2}\right\} \\ 
        & \frac{V^{S}_u(9,3) +  V^{S}_u(9,1)}{2} = \frac{1}{6} + \frac{0.1642}{2}=0.2488 < \frac{2}{8}.
    \end{align*}
    Hence we have $V^{S}(8,2)\leq g(8,2) = \frac{2}{8}$ and it follows that $(8,2)$ is in the stopping set.\footnote{For a detailed description of the method see Section \ref{sec:comp}.}.
\end{example}

\subsection{Generalizations}\label{ssec:gen}
In the proof of Theorem \ref{th:up} we did not use the specific form of the gain function $g(t,x) = \frac{x}{t}$. Everything we needed was that:
\begin{itemize}
    \item The value function of the stopping problem 
    \begin{equation}\label{eq:vwgen}
    V^{W}(t,x) = \sup_{\tau  }\e \left [ g(t+\tau,x + W_\tau)\right]
    \end{equation}
    is on $C$ of the form 
    \[V^{W}|_{C}(t,x) = \int_{\RR}e^{ax - \frac{1}{2}a^2t}\dif \mu(a), \]
    for a measure $\mu$.
    \item The function 
    \[h(t,x) = \int_{\RR}e^{ax - \frac{1}{2}a^2t}\dif \mu(a) \]
    dominates $g$ on $\RR^{+}\times\RR$.
    \item The boundary of the continuation set $\partial C^{S}$ of the stopping problem
    \[V^{S}(t,x) = \sup_{\tau  }\e \left [ g(t+\tau,x + S_\tau)\right]\]
    is the graph of a function $b:\RR^{+}\to \RR$.\\
    This requirement can easily be relaxed to the symmetric case, where $\partial C^{S} = \mathrm{Graph}(b) \cup \mathrm{Graph}(-b)$.
\end{itemize}
These requirements are not very restrictive, and there are many other gain functions and associated stopping problems for which this kind of upper bound can be constructed. A set of examples which fulfill these requirements and for which \eqref{eq:vwgen} is explicitly solvable can be found in \cite{peskir06}. Some of these are:
    \[g(t,x) = \frac{x^{2d-1}}{t^{q}}\]
with $d\in \NN$ and $q>d-\1$,
    \[g(t,x) = |x|-\beta\sqrt{t}\]
for some $\beta\geq 0$, and
    \[g(t,x) = \frac{|x|}{t}.\]

\section{A lower bound for $V^{S}$}\label{sec:low}
In this section we want to give a lower bound for the value function of the Chow-Robbins game \eqref{eq:cr}. Here $S$ will always be a symmetric Bernoulli random walk.
The basis of our construction is the following lemma.

    \begin{lemma}[A lower bound for $V$]\label{lem:superm}
    Let $X$ be a random walk, $g$ be a gain function,
    $V(t,x) = \sup_{\tau  }\e g(\tau+t,X_\tau+x)$ and $h:\RR^{+}\times \RR \to \RR$ measurable. For a given point $(t_0,x_0)$ let $\tau$ be a stopping time, such that 
    the stopped process 
    $(h(t\wedge \tau+ t_0,X_{t\wedge \tau}+ x_0))_{t\geq 0}$ is a submartingale and
        \[h(\tau + t_0,X_\tau + x_0)\leq g(\tau + t_0,X_\tau + x_0) \text{ a.s.}\]
    Then
        \[h(t_0,x_0)\leq V(t_0,x_0).\]
\end{lemma}

\begin{proof}
    Since $h(\tau,X_\tau)\leq g(\tau,X_\tau)$ we can use the optional sampling theorem and obtain
        \[h(t_0,x_0) \leq \e [h(\tau+t_0,X_\tau+x_0)] \leq \e [g(\tau+t_0,X_\tau+x_0)] \leq V(t_0,x_0).\]
\end{proof}

We modify the function $h(t,x)=(1-\alpha^2)\int_{0}^{\infty}e^{ax-\frac{a^2}{2}t} \dif a$ from Lemma \ref{lem:wt} slightly to
\begin{equation}\label{eq:hc}
    h_{c}(t,x) := K  \int_{0}^{\infty}e^{ax-\frac{c}{2}a^2t} \dif a
\end{equation}
for some $0<c< 1$ and $K>0$ to mach the assumptions of Lemma \ref{lem:superm}. As a stopping time we choose 
\begin{equation}
    \tau_0 = \inf \{n\geq 0\mid x+S_n \geq \alpha \sqrt{t+n}-1\}.
\end{equation}
Unfortunately there is no $c$ such that \eqref{eq:hc} is globally $S$-subharmonic hence, we have to choose $c$ depending on the time horizon $T$. This makes the following result a bit technical.

\begin{theorem}[A lower bound for $V^{S}$]\label{th:low}
    Let  
        \[h_c(t,x):=K  \int_{0}^{\infty}e^{ax-\frac{c}{2}a^2t} \dif a 
        = K\frac{1}{ct}\frac{\Phi_{ct}(x)}{\varphi_{ct}(x)}\]
    with 
        \[K 
        = \alpha c \frac{\varphi_{c}(\alpha)}{\Phi_{c}(\alpha)}.\]
    Given a time horizon $T>0$, let $c_1$ be the biggest solution smaller than $1$  to
    \begin{equation}
        \frac{1}{2}\left(h_{c}(T+1,\alpha\sqrt{T}-1)+h_{c}(T+1,\alpha\sqrt{T}+1)\right) = h_{c}(T,\alpha\sqrt{T}),
    \end{equation}
    $c_2$ the unique positive solution of
    \begin{equation}\label{eq:c2}
        h_c(T,\alpha\sqrt{T}-1)=\frac{\alpha\sqrt{T}-1}{T},
    \end{equation}
    and $c=\min\{c_1,c_2\}$.
    Let $a_0$ be the unique positive solution (in $a$) to
    \[\frac{1}{2}\left(e^{a}+e^{-a}\right)e^{-\frac{c}{2} a^2}=1.\]
    If 
    \begin{equation}\label{eq:t2}
        T\geq\left(\frac{\alpha}{a_0c}\right)^2,
    \end{equation}
    then 
    $h_c(t,x)$ is a lower bound for $V^{S}(t,x)$ for all $t\geq T$ and $x\leq \alpha \sqrt{t}$.
\end{theorem}
 
\begin{remark}
    Our numerical evaluations suggest that for $T\geq 4$ \eqref{eq:t2} is always satisfied and for $T\geq 200$ we always have $c=c_1$.
\end{remark}

\begin{proof}
    We divide the proof into tree parts:\\
    (1.) We show that the stopped process $h_c(t\wedge\tau_0,x+S_{t\wedge\tau_0})_{t\geq T}$ is a submartingale.\\
    (2.) We calculate $K$.\\
    (3.) We show that 
        \[h(\tau_0,x+S_{\tau_0})\leq g(\tau_0,x+S_{\tau_0}) ~P_{T,x}\text{-a.s.},\]
    and use Lemma \ref{lem:superm} to prove the statement. An illustration of the setting is given in Figure \ref{hvsV}.\\

    (1.) We have to show that 
    \begin{equation}\label{eq:subm2}
        \frac{1}{2}\left(h_{c}(t+1,x-1)+h_{c}(t+1,x+1)\right) \geq h_{c}(t,x)
    \end{equation}
    for every $t\geq T$ and $x\leq \alpha \sqrt{t}-1$ and will even show \eqref{eq:subm2} for all $x\leq \alpha \sqrt{t}$.\\
    The constant $K$ has no influence on \eqref{eq:subm2}, so we set it equal 1 for now. 
    We have
    \begin{align}
        f_c(t,x)&:=\frac{1}{2}\big(h_{c}(t+1,x-1)+h_{c}(t+1,x+1)\big) - h_{c}(t,x)\nonumber\\
        &=\int_{0}^{\infty}\frac{1}{2} e^{a(x+1)-\frac{c}{2}a^2(t+1)}\dif a +  
        \int_{0}^{\infty}\frac{1}{2} e^{a(x-1)-\frac{c}{2}a^2(t+1)}\dif a - \int_{0}^{\infty}e^{ax-\frac{c}{2}a^2t} \dif a \nonumber\\
        &= \int^{\infty}_{0}e^{ax-\frac{c}{2}a^2t}\left[\frac{1}{2}(e^a+e^{-a})e^{-\frac{c}{2}a^2}-1 \right ]\dif a.\label{eq:fc1}
    \end{align}
    The function $\lambda(a):=\frac{1}{2}(e^a+e^{-a})e^{-\frac{c}{2}a^2}-1$ has a unique positive root $a_0$ and for $a\in[0,a_0]$ we have $\lambda(a)\geq 0$ and for $a\geq a_0$ $\lambda(a)\leq 0$.\\
    Suppose for given $(t,x)$, we have $f_c(t,x)\geq 0$. Let $\delta \geq 0$, $\varepsilon \in \RR$, we have
    \begin{align*}
        & f_c(t+\delta,x+\varepsilon)  = \int^{a_0}_{0}e^{ax-\frac{c}{2}a^2t} \lambda(a) e^{\varepsilon a - \delta \frac{c}{2} a^2} \dif a 
        + \int^{\infty}_{a_0}e^{ax-\frac{c}{2}a^2t} \lambda(a) e^{\varepsilon a - \delta \frac{c}{2} a^2} \dif a \\
        & \overset{(\ast \ast)}\geq \int^{a_0}_{0}e^{ax-\frac{c}{2}a^2t} \lambda(a) e^{\varepsilon a_0 - \delta \frac{c}{2} a_0^2} \dif a 
        + \int^{\infty}_{a_0}e^{ax-\frac{c}{2}a^2t} \lambda(a) e^{\varepsilon a_{0} - \delta \frac{c}{2} a_0^2} \dif a\\
        &= e^{\varepsilon a_{0} - \delta \frac{c}{2} a_0^2}  f_c(t,x) \geq 0.
    \end{align*}
    Here $(\ast\ast)$ is true if $\varepsilon a - \delta \frac{c}{2} a^2 \geq \varepsilon a_{0} - \delta \frac{c}{2} a_0^2$ for $a\leq a_{0}$ and
    $\varepsilon a - \delta \frac{c}{2} a^2 \leq \varepsilon a_{0} - \delta \frac{c}{2} a_0^2$ 
    for $a\geq a_{0}$, what is the case if
    \begin{equation}\label{eq:eps}
        \varepsilon \leq a_0 \delta \frac{c}{2}.
    \end{equation}
    By assumption, we have $f_c(T,\alpha \sqrt{T})\geq 0$. (If $c=c_1 $ as in all our computational examples, this is clear. If $c=c_2<c_1$ an inspection of $f_c$ in \eqref{eq:fc1} shows that $f_c>f_{c_1}$.)
    The function $\alpha \sqrt{t}$ is concave and 
        \[\frac{\partial}{\partial t}\alpha\sqrt{t}=\frac{\alpha}{2\sqrt{t}},\]
    so for $t\geq T$ and $x\leq \alpha\sqrt{t}$ with $(t,x)=(T+\delta,\alpha \sqrt{T}+\varepsilon)$ we have 
        \[\varepsilon\leq \delta\frac{\alpha}{2\sqrt{T}}.\]
    Putting this into \eqref{eq:eps} we get the condition
        \[\delta\frac{\alpha}{\sqrt{T}}\leq a_0 \delta c,\text{ i.e.\ }  T\geq\left(\frac{\alpha}{a_0c}\right)^2\]
    what is true by assumption. That concludes the first part of the proof.\\

    (2.) We want to choose $K$ such that $h_c(t,\alpha\sqrt{t})=g(t,\alpha\sqrt{t}) = \frac{\alpha}{\sqrt{t}}$. We first show that this is possible and then calculate $K$.
    We have
    \begin{align*}
        & h_c(t,x ) = K  \int_{0}^{\infty}e^{ax-\frac{c}{2}a^2t} \dif a = 
        K\frac{1}{ct}\frac{\Phi_{ct}(x)}{\varphi_{ct}(x)}
    \end{align*}
    and
    \begin{align*}
        & h_c(t,\alpha\sqrt{t}) = 
        K\frac{1}{ct}\frac{\Phi_{ct}(\alpha\sqrt{t})}{\varphi_{ct}(\alpha\sqrt{t})}
        = K\frac{1}{c\sqrt{t}}\frac{\Phi_{c}(\alpha)}{\varphi_{c}(\alpha)}
    \end{align*}
    what depends only on $\frac{1}{\sqrt{t}}$. Solving $h_c(t,\alpha\sqrt{t}) = \frac{\alpha}{\sqrt{t}}$ we get
        \[K =   \alpha c \frac{\varphi_{c}(\alpha)}{\Phi_{c}(\alpha)}.\]

    \begin{figure}[ht]
        \centering
        \includegraphics[width=0.7\textwidth]{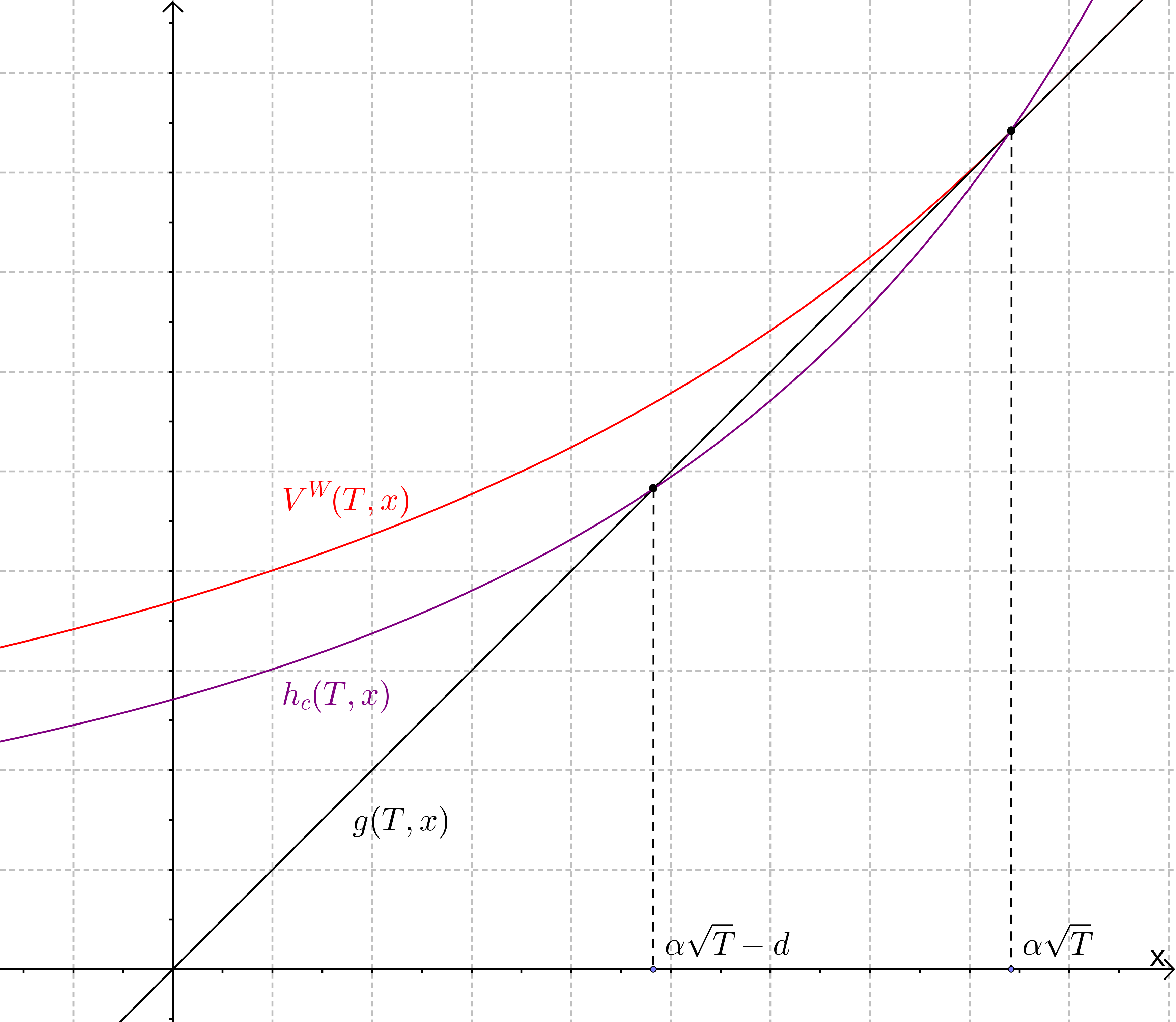}
        \caption{The upper bound $V^{W}$ and the lower bound $h_c$ for a fixed $T$. For a better illustration $c=0.6$ is chosen very small.}
        \label{hvsV}
    \end{figure}

    (3.) We chose $\tau_0 = \inf \{n\geq 0\mid x+S_n \geq \alpha \sqrt{t+n}-1\}$ and need to show that $h_c(\tau_0,S_{\tau_0})\leq g(\tau_0,S_{\tau_0})$. 
    It is clear that $S_{\tau_0}\in [\alpha\sqrt{\tau_0}-1,\alpha\sqrt{\tau_0}]$. 
    By the construction of $K$ in (2.) we know that $h_c(t,\alpha\sqrt{t})= g(t,\alpha\sqrt{t})$. 
    Since $h_c$ has strictly positive curvature, we know that $h(t, \cdot)$ has exactly one more intersection with $g(t,\cdot)$ which we denote by $\alpha \sqrt{t}-d(t)$. We will see that $d(t) \geq 1$ for $t\geq T$ and hence $h_c(t,x)\leq g(t,x)$ for $x\in [\alpha\sqrt{t}-d(t),\alpha\sqrt{t}]$.
    We have seen in (2.) that $\sqrt{t}\cdot h_c(t,\beta\sqrt{t})$ is constant for any $\beta>0$. If for some $x_0$
        \[h_c(T,x_0)\leq g(T,x_0)=\frac{x_0}{T},\]
    we set $\beta :=\frac{x_0}{\sqrt{T}}$ and see that for all $t>T$ we have
        \[h(t,\frac{x_0}{\sqrt{T}}\sqrt{t})\leq g(t,\frac{x_0}{\sqrt{T}}\sqrt{t})=\frac{1}{\sqrt{t}}\frac{x_0}{\sqrt{T}}.\]
    If $x_0\leq \alpha\sqrt{T}-1$, then $ \frac{x_0}{\sqrt{T}}\sqrt{t}\leq \alpha\sqrt{t}-1$.
    If $d(T) \geq 1 $ then we set $x_0 := \alpha \sqrt{T}-d(T)$. We can now conclude that for $t \geq T$ we have $d(t) \geq 1$, hence it is enough to show that $d(T)\geq 1$. For $c=c_2$ this is true by assumption. In general $c\leq c_2$ and we have $\frac{\partial }{\partial x}h_c(t,x) \leq \frac{\partial }{\partial x}h_{c_2}(t,x)$. Since $h_{c_2}(T,\alpha \sqrt{T})= h_{c_2}(T,\alpha \sqrt{T})$ we have that
        \[h_{c}(T,\alpha \sqrt{T}-1)\leq h_{c_2}(T,\alpha \sqrt{T}-1)=\frac{\alpha \sqrt{T}-1}{T}\]
    and the statement follows. 
    Now $h_c$ and $\tau_0$ fulfill the conditions of Lemma \ref{lem:superm}. This completes the proof.
\end{proof}
 
\begin{remark}
    The only properties of $S$ we used in the proof, are that $S$ has limited jump sizes upwards and that
        \[m_{S_1}(a)=\e [e^{a S_1}])\]
    has only one positive intersection with $e^{\frac{c}{2}a^2}$ (i.e.\ $\frac{1}{2}( e^{a} + e^{-a})e^{-\frac{c}{2}a^2}-1$ has only one positive root).
    This kind of lower bound can be constructed for any random walk with increments that fulfill these two conditions. This would of course result in different values for $c$. 
\end{remark}

Some values for $c$ are given in the table below
\begin{center}
\begin{tabular}{ll}
    \textbf{T} & \textbf{c}\\
    $10^{3}$ & 0.999204\\
    $10^{4}$ & 0.9999212\\
    $10^{5}$ & 0.99999214\\
    $10^6$   & 0.999999216.
\end{tabular}
\end{center}

\section{Error of the bounds}
We want to show that the relative error of the constructed bounds is of order $\mathcal{O}(1/T)$, for $x\geq 0$. First, we show that $c = c(T) \geq \frac{T-1}{T}$ for $T$ large enough. Indeed, for $c_1$ evaluating \eqref{eq:fc1} for $c = \frac{T}{T+1}$ yields
\begin{align}
    f_{\frac{T}{T+1}}(T,x) &= \int^{\infty}_{0}e^{ax-\frac{1}{2}a^2T}\left[\frac{1}{2}(e^a+e^{-a})-e^{\frac{1}{2}\frac{1}{T+1}} \right ]\dif a,
\end{align}
which can be seen to be positive for $T$ large enough, yielding $c_1\geq \frac{T}{T+1}\geq \frac{T-1}{T}$. We evaluate \eqref{eq:c2} for $c = \frac{T-1}{T}$ and get with elementary estimates
\begin{align*}
    & h_{\frac{T-1}{T}}(T,\alpha\sqrt{T}-1)
    = \frac{K}{\sqrt{T-1}}\frac{\Phi\left(\frac{\alpha \sqrt{T}-1}{\sqrt{T}}\right)}{\varphi\left(\frac{\alpha \sqrt{T}-1}{\sqrt{T}}\right)}\\
    & \leq \frac{K}{\sqrt{T-1}}\left(\frac{\Phi\left(\frac{\alpha \sqrt{T}}{\sqrt{T}}\right) - \frac{1}{\sqrt{T-1}}\varphi\left(\frac{\alpha \sqrt{T}}{\sqrt{T}}\right)}{\varphi\left(\frac{\alpha \sqrt{T}}{\sqrt{T}}\right)e^{\frac{\alpha\sqrt{T}-\frac{1}{2}}{T-1}}}\right)\\
    & = e^{-\frac{\alpha\sqrt{T}-\frac{1}{2}}{T-1}}\left(\frac{\alpha }{\sqrt{T}}  - \frac{K}{T-1}\right) \leq \frac{\alpha \sqrt{T}-1}{T}
\end{align*}
for $T$ large enough, so that $c_2 \geq \frac{T-1}{T}$. We obtain
$c = \min \{c_1,c_2\}\geq \frac{T-1}{T}$.
We now calculate the asymptotic relative error between $V_{u}$ and $V_{l}$ in $x = 0$:
\begin{align*}
    & \frac{V_u(T,0)}{V_l(T,0)}-1 
    = \frac{h(T,0)}{h_c(T,0)}-1 = \frac{1-\alpha^2}{K}\sqrt{c}-1
    =\frac{1-\alpha^2}{\alpha}\frac{\Phi\left(\frac{\alpha}{\sqrt{c}}\right)}{\varphi\left(\frac{\alpha}{\sqrt{c}}\right)}-1
\end{align*}
We approximate
\begin{align*}
    &\frac{\Phi\left(\frac{\alpha}{\sqrt{c}}\right)}{\varphi\left(\frac{\alpha}{\sqrt{c}}\right)}
    \approx e^{-\frac{\alpha^2(c-1)}{2c}}\left(\frac{\Phi(\alpha)}{\varphi(\alpha)} + \frac{\alpha(1-\sqrt{c})}{\sqrt{c}}\right)
    = e^{-\frac{\alpha^2(c-1)}{2c}}\left(\frac{\alpha}{1-\alpha^2} + \frac{\alpha(1-\sqrt{c})}{\sqrt{c}}\right)
\end{align*}
and get with $c = \frac{T-1}{T}$
\begin{align*}
    & \frac{h(T,0)}{h_c(T,0)}-1
    \approx e^{-\frac{\alpha^2(c-1)}{2c}}\left(1+\frac{\alpha(1-\sqrt{c})}{\sqrt{c}}\right) -1\\
    & =e^{\frac{\alpha^2}{2T-2}}\left(1+\alpha\left(\sqrt{\frac{T-1}{T}}-1\right)\right)-1
    = \mathcal{O}\left(\frac{1}{T}\right).
\end{align*}
It is now straightforward to check that for $\alpha \sqrt{T} \geq x\geq 0$
    \[\frac{h(T,0)}{h_c(T,0)}-1 \geq \frac{h(T,x)}{h_c(T,x)}-1. \]
This yields that
    \[\frac{V(T,x)}{V_l(T,x)}-1 =\mathcal{O}\left(\frac{1}{T}\right) \text{ and }  \frac{V(T,x)}{V_u(T,x)}-1 =\mathcal{O}\left(\frac{1}{T}\right),\]
for all $\alpha \sqrt{T} \geq x\geq 0$.

\section{Computational results}\label{sec:comp}
In this section we show how to compute the continuation and stopping set for the Chow-Robbins game. 
In 2013 H\"aggstr\"om and W\"astlund \cite{haeggstoem2013} computed stopping and continuation points starting from $(0,0)$. They choose a, rather large, time horizon $T=10^{7}$, and set \footnote{They use another unsymmetric notation of the problem. We give their bounds transformed into our setting (see appendix).}
    \[V^S_{l}(T,x)= \max\left\{\frac{x}{T},0\right\}\]
as a lower and
    \[V^S_{u}(T,x)= \max\left\{\frac{x}{T},0\right\} + \min\left\{\sqrt{\frac{\pi}{T}},\frac{1}{|x|}\right\}\]
as an upper bound. Then they use backward induction to calculate $V^S(n,x)$, for $n<T$ and $i\in\{u,l\}$ with
\begin{equation}\label{eq:bw}
    V_{i}^S(n,x) = \max \left \{ \frac{x}{n}, \e[V_i^{S}(n+1,x+X_i)] \right \}.
\end{equation}
If $V_u(n,x)=\frac{x}{n}$ then $(n,x)\in D$, if $V_l(n,x)>\frac{x}{n}$ then $(n,x)\in C$. In this way they were able to decide for all but 7 points $(n,x)\in\NN\times\ZZ$ with $n\leq 1000$, if they belong to $C$ or $D$.\\
We use backward induction from a finite time horizon as well, but use the much sharper bounds given in Section \ref{sec:up} and \ref{sec:low}. For our upper bound this has a nice intuition. We play the Chow-Robbins game up to the time horizon $T$, then we change the game to the favorable $\frac{W_t}{t}$-game, what slightly rises our expectation.\\

With a time horizon $T=10^{6}$ we are able to calculate all stopping and continuation points $(n,x)\in\NN\times\ZZ$ with $n\leq 489.241$. We show that all open points in \cite{haeggstoem2013} belong to $D$.\\

\paragraph{Description of the method}
Unlike H\"aggstr\"om and W\"astlund we use the symmetric notation. Let $X_i$ be iid.\ random variables with $P(X_i =-1) =P(X_i =1) = \frac{1}{2} $ and $S_n=\sum^{n}_{i=1} X_i$. 
We choose a time horizon $T$ and use $V^{W}$ given in Lemma \ref{lem:wt} as an upper bound
    \[V^S_{u}(T,x)= V^{W}(T,x)\]
and $h_c$ given in Theorem \ref{th:low} as a lower bound
    \[V^S_{l}(T,x)= h_c(T,x),\]
for $x\in \ZZ$ with $x\leq \alpha \sqrt{T}$.
For $i\in\{u,l\}$ we now calculate recursively
    \[V_i^S(n,x) = \max \left \{ \frac{x}{n}, \frac{V_i^S(n+1,x+1) +  V_i^S(n+1,x-1)}{2} \right \}.\]
If $V_l(n,x)>\frac{x}{n}$, then $(n,x)\in C$. To check if $(n,x)\in D$ we use, instead of $V_u(n,x)=\frac{x}{n}$, the slightly stronger, but numerically easier to evaluate, condition $(n,x)\in D$ if 
\[\frac{V_u^S(n+1,x+1) +  V_u^S(n+1,x-1)}{2}< \frac{x}{n}.\]
We use $T=10^{6}$ to calculate $V(0,0)$ and the integer thresholds $\hat b(n) := \lceil b(n) \rceil$. For 34 values $n\leq 10^{6}$ the exact value $\hat b(n)$ can not be determined this way, the smallest such value is $n = 489.242$. 

\begin{theorem}\label{th:calc}
    For the stopping problem \eqref{eq:cr} starting in $(0,0)$ the stopping boundary $\hat b$  is for $n\leq 489.241$ given by
    \begin{equation}
        \hat b(n)=\left \lceil \alpha\sqrt{n} - \1 + \frac{1}{7.9+4.54\sqrt[4]{n}} \right \rceil
    \end{equation}
    with the following 8 exceptions:
    \begin{center}
    \begin{tabular}{rr|rr|rr|rr}
        \textbf{n} & $\hat b(n)$ & \textbf{n} & $\hat b(n) $& \textbf{n} & $\hat b(n)$ & \textbf{n} & $\hat b(n)$\\
        $3195$ & $48$ & $14312$ & $101$ & $25257$ & $134$ & $51434$ & $191$ \\
        $12923 $&$ 96 $& $24880 $& $133 $& $44653 $& $178 $& $116342 $& $287$
    \end{tabular}
    \end{center}
    For the value function we have
        \[0.5859070128172 \leq V^{S}(0,0) \leq 0.5859070128182.\footnote{
        H\"aggstr\"om and W\"astlund use a different notation (denoted with a ' here), with $P(X'_i =0) =P(X'_i =1) = \frac{1}{2}$. Our functions and values translate to $ V(n,x) = 2V'(n,\frac{x+n}{2})-1$,
        $b'(n)=\left \lceil \frac{\alpha\sqrt{n}+n}{2} -\rho'(n)  \right \rceil$ with $\rho'(n) = \frac{1}{4} - \frac{1}{15.8+9\sqrt[4]{n}}$, and $0.7929535064086 \leq V'(0,0) \leq 0.7929535064091$. The value of $V(0,0)$ is calculated with $T=2\cdot10^{6}$.}\]
\end{theorem}

\begin{remark}
    The function $(7.9+4.54\sqrt[4]{n})^{-1} $ is constructed from our computed data. It is an interesting question whether it is indeed possible to show that $\alpha \sqrt{t} -  b(t) = \1 -\mathcal{O}(t^{-\frac{1}{4}})$.
    Lai, Yao and AitSahlia introduced a method to show that 
        \[\lim_{t\to \infty}  \alpha \sqrt{t} -  b(t) = \frac{1}{2}\]
    in \cite{lai2007}.
    This is reflected nicely in our calculations.
\end{remark}

\begin{figure}[ht]
    \centering
    \includegraphics[width=0.7\textwidth]{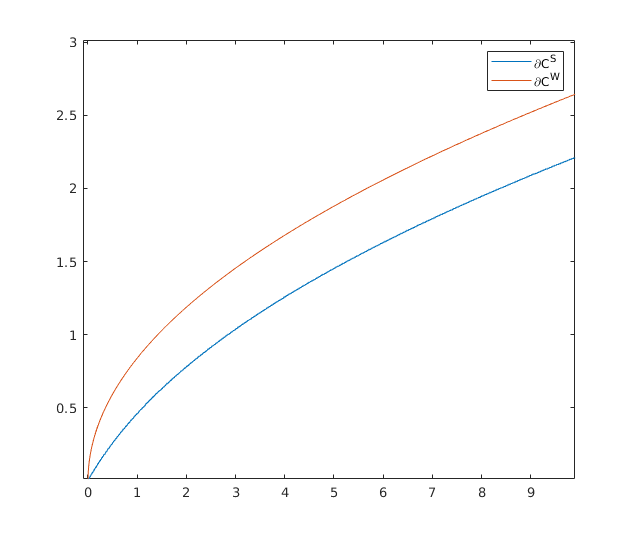}
    \caption{The boundaries of the continuation sets. While asymptotically similar, they behave differently close to 0.}
    \label{fig:1}
\end{figure}

\begin{figure}[ht]
    \centering
    \includegraphics[width=0.7\textwidth]{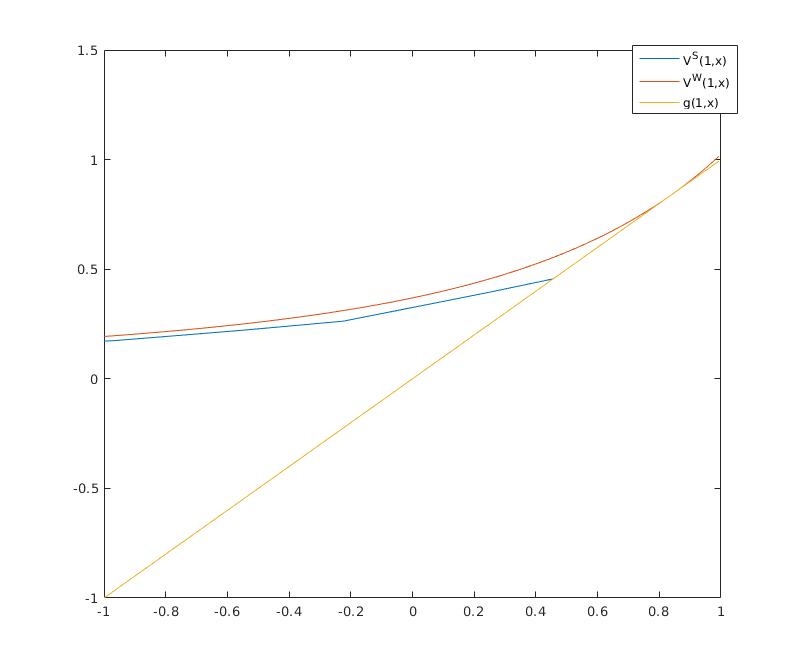}
    \caption{The value functions $V^{W}$ and $V^{S}$ in $t=1$. $V^{S}$ doesn't follow the smooth fit principle and is not everywhere smooth on $C$.}
    \label{fig:2}
\end{figure}

\begin{figure}[ht]
    \centering
    \includegraphics[width=0.7\textwidth]{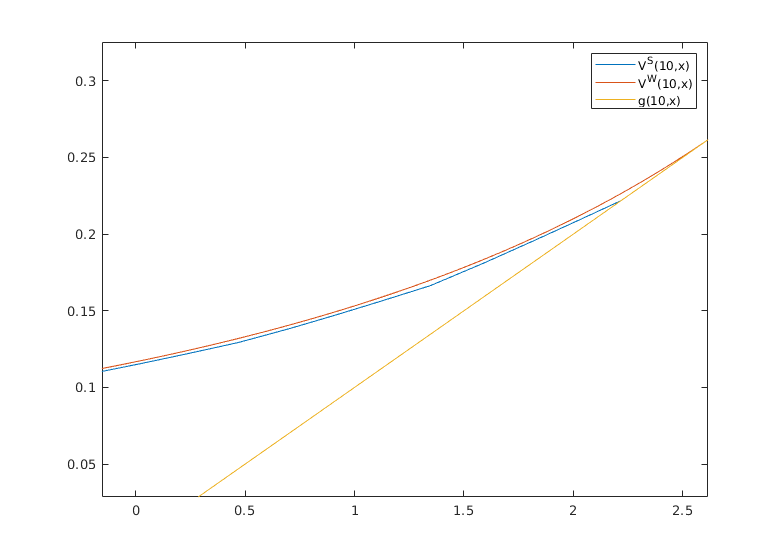}
    \caption{The value functions $V^{W}$ and $V^{S}$ in $t=10$. }
    \label{fig:3}
\end{figure}

We calculated $V^{S}$ and $b$ as well for non-integer values. We did this by choosing an integer $D$ and then calculate $V^{S}$ on 
$\frac{1}{D}\NN\times\frac{1}{D}\ZZ$ with the method described above.\footnote{For most plots we used $D=300$.} Some plots of $b$ and $V^{S}$ are given in Figures \ref{fig:1} - \ref{fig:3}. The method enables us to get very detailed impressions of $b$ and $V^{S}$ what inspires further analytical research. In $\cite{christensen2020}$ the authors showed that $V^S$ is not differentiable on a dense subset of $C^S$ and that $C^S$ is not convex.

\cleardoublepage
  
\bibliographystyle{acm}
\bibliography{On_the_Sn_problem}

\begin{thebibliography}{10}

\bibitem{chow1965}
{\sc Chow, Y.~S., and Robbins, H.}
\newblock On optimal stopping rules for $s_{n}/n$.
\newblock {\em Illinois J. Math. 9}, 3 (09 1965), 444--454.

\bibitem{christensen2020}
{\sc Christensen, S., and Fischer, S.}
\newblock Note on the (non-)smoothness of discrete time value functions in
  optimal stopping.
\newblock {\em Electron. Commun. Probab. 25\/} (2020), 10 pp.

\bibitem{dvoretzky1967}
{\sc Dvoretzky, A.}
\newblock Existence and properties of certain optimal stopping rules.
\newblock In {\em Proceedings of the Fifth Berkeley Symposium on Mathematical
  Statistics and Probability, Volume 1: Statistics\/} (Berkeley, Calif., 1967),
  University of California Press, pp.~441--452.

\bibitem{haeggstoem2013}
{\sc H\"aggstr\"om, O., and W\"astlund, J.}
\newblock Rigorous computer analysis of the chow-robbins game.
\newblock {\em The American Mathematical Monthly 120}, 10 (2013), 893--900.

\bibitem{lai2006}
{\sc Leung~Lai, T., and Yao, Y.-C.}
\newblock The optimal stopping problem for $\frac{S_n}{n}$ and its
  ramifications.
\newblock {\em Technical reports, Department of statistics, Stanford
  University}, 2005-22 (01 2005).

\bibitem{lai2007}
{\sc Leung~Lai, T., Yao, Y.-C., and Aitsahlia, F.}
\newblock Corrected random walk approximations to free boundary problems in
  optimal stopping.
\newblock {\em Advances in Applied Probability 39\/} (09 2007), 753--775.

\bibitem{medina2009}
{\sc {Medina}, L.~A., and {Zeilberger}, D.}
\newblock {An Experimental Mathematics Perspective on the Old, and still Open,
  Question of When To Stop?}
\newblock {\em arXiv e-prints\/} (June 2009), arXiv:0907.0032.

\bibitem{peskir06}
{\sc Peskir, G., and Shiryaev, A.}
\newblock {\em Optimal Stopping and Free-Boundary Problems}.
\newblock Birkh\"auser Basel, 2006.

\bibitem{shepp1969}
{\sc Shepp, L.~A.}
\newblock Explicit solutions to some problems of optimal stopping.
\newblock {\em Ann. Math. Statist. 40}, 3 (06 1969), 993--1010.

\bibitem{walker1969}
{\sc Walker, L.~H.}
\newblock Regarding stopping rules for brownian motion and random walks.
\newblock {\em Bulletin Amer. Math. Soc.}, 75 (1969), 46--50.

\end{thebibliography}
  
\end{document}